\documentclass[a4paper, 11pt, oneside]{scrartcl}
  \usepackage[english]{babel}
  \usepackage[utf8]{inputenc}
  \usepackage{paralist}
  \usepackage{mathtools}
  \usepackage{amssymb}
  \usepackage{amsthm}
  \usepackage{bbold}
  \usepackage{microtype}
  \usepackage{enumerate}
  \usepackage{mathrsfs}
  \usepackage{xcolor}
  \usepackage{graphicx}
  \usepackage{paralist}
  \usepackage[normalem]{ulem}
  \usepackage{tikz}
  \usetikzlibrary{matrix,arrows,positioning,shapes}
  \usepackage[square,numbers]{natbib}
  \usepackage{hyperref}
  \usepackage[capitalise]{cleveref}
  \usepackage{authblk}

  \theoremstyle{plain}
  \newtheorem{theorem}{Theorem}
  \newtheorem{lemma}[theorem]{Lemma}
  \newtheorem{prop}[theorem]{Proposition}
  \newtheorem{corollary}[theorem]{Corollary} 
    
  \theoremstyle{definition}
  \newtheorem{defn}[theorem]{Definition}
  
  \theoremstyle{remark}

  \usepackage[left=3cm,right=3cm,top=3.5cm,bottom=4.2cm]{geometry}
  
  \DeclarePairedDelimiter{\FP}{\{}{\,|\kern-0.2em \}}
  \DeclarePairedDelimiter{\FK}{[}{\,|\kern-0.2em ]}

  \newcommand{\K}{\mathbb{K}}
  \newcommand{\F}{\mathbb{F}}

  \newcommand{\D}{\text{D}}

  \renewcommand{\phi}{\varphi}

  \newcommand{\spa}{\operatorname{span}}

  \newcommand{\factor}[2]{\left.\raisebox{.2em}{$#1$}\middle/\raisebox{-.2em}{$#2$}\right.}

  \newcommand{\Hom}{\operatorname{Hom}}

  \let\temp\phi
  \let\phi\varphi
  \let\varphi\temp

  \newcommand{\operad}[1]{\mathbf{#1}} 
 \newcommand{\lLeib}{\mathbf{Leib}}
 \newcommand{\rLeib}{\mathbf{rLeib}}
 \newcommand{\sLeib}{\mathbf{sLeib}}
 \newcommand{\sPerm}{\mathbf{sPerm}}
 \newcommand{\Lie}{\mathbf{Lie}}
 \newcommand{\Ass}{\mathbf{Ass}} 
 \newcommand{\Com}{\mathbf{Com}}
 \newcommand{\Alg}[1]{\operatorname{\mathbf{#1}-Algs}}
 \newcommand{\OP}{\mathcal{P}}
 \newcommand{\OO}{\mathcal{O}}

 \newcommand{\lievert}[1]{%
	\begin{tikzpicture}[scale=0.3,lie/.style={draw,shape=rectangle,minimum size=5,inner sep=0}, com/.style={draw,shape=diamond,minimum size=5,inner sep=0}]
	 \node [lie] (a0) at ( #1/2 + 0.5,1) {};
	 \draw (a0) -- +(0,0.7) ;
	 \foreach  \i in{1,...,#1}
	 {
	\draw (\i,0) -- (a0);
	}
	\end{tikzpicture}%	
 }

 \newcommand{\comvert}[1]{%
	\begin{tikzpicture}[scale=0.3,lie/.style={draw,shape=rectangle,minimum size=5,inner sep=0}, com/.style={draw,shape=diamond,minimum size=5,inner sep=0}]
	 \node [com] (a0) at ( #1/2 + 0.5,1) {};
	 \draw (a0) -- +(0,0.7) ;
	 \foreach  \i in{1,...,#1}
	 {
	\draw (\i,0) -- (a0);
	}
	\end{tikzpicture}%	
 }

 \newcommand{\verte}[2]{%
	\begin{tikzpicture}[scale=0.3,dot/.style={draw,shape=circle,fill=blue,minimum size=2,inner sep=0},lie/.style={draw,shape=rectangle,minimum size=5,inner sep=0}, com/.style={draw,shape=diamond,minimum size=5,inner sep=0},tex/.style={draw, shape=circle,inner sep=0}]
	 \node [tex] (a0) at ( #1/2 + 0.5,1) {\tiny #2};
	 \draw (a0) -- +(0,0.7) ;
	 \foreach  \i in{1,...,#1}
	 {
	\draw (\i,0) -- (a0);
	}
	\end{tikzpicture}%	
 }

 \newcommand{\tree}[2]{%
 \begin{tikzpicture}[baseline=5,scale=0.3,dot/.style={draw,shape=circle,fill=blue,minimum size=2,inner sep=0},lie/.style={draw,shape=rectangle,minimum size=4,inner sep=0}, com/.style={draw,shape=diamond,minimum size=5,inner sep=0}]
	\node [#1] (a0) at (2.25,2) {};
	\node [#2] (a1) at (1.5,1) {};
	\draw (a0) -- +(0,0.7) ;
	\draw (1,0) -- (a1);
	\draw (2,0) -- (a1);
	\draw  (a1) -- (a0); 
	\draw (3,0) -- (a0); 
	\end{tikzpicture}%	
    }
    \newcommand{\treete}[2]{%
    \begin{tikzpicture}[baseline=5,scale=0.3,dot/.style={draw,shape=circle,fill=blue,minimum size=2,inner sep=0},lie/.style={draw,shape=rectangle,minimum size=4,inner sep=0}, com/.style={draw,shape=diamond,minimum size=5,inner sep=0},tex/.style={draw, shape=circle,inner sep=0}]
       \node [tex] (a0) at (2.25,2) {\tiny #1};
       \node [tex] (a1) at (1.5,1) {\tiny #2};
       \draw (a0) -- +(0,0.7) ;
       \draw (1,0) -- (a1);
       \draw (2,0) -- (a1);
       \draw  (a1) -- (a0); 
       \draw (3,0) -- (a0); 
       \end{tikzpicture}%	
       }

    \newcommand{\treerte}[2]{%
    \begin{tikzpicture}[baseline=5,scale=0.3,dot/.style={draw,shape=circle,fill=blue,minimum size=2,inner sep=0},lie/.style={draw,shape=rectangle,minimum size=4,inner sep=0}, com/.style={draw,shape=diamond,minimum size=5,inner sep=0},tex/.style={draw, shape=circle,inner sep=0}]
       \node [tex] (a0) at (1.75,2) {\tiny #1};
       \draw (a0) -- +(0,0.7) ; 	
       \draw (1,0) -- (a0); 
       \node [tex] (a2) at (2.5,1) {\tiny #2};
       \draw (2,0) -- (a2);
       \draw (3,0) -- (a2);
       \draw  (a2) -- (a0); 
       \end{tikzpicture}%	
       }

 \allowdisplaybreaks

  \title{A diagram connecting Symmetric Leibniz algebras with Lie-admissible algebras}
 \author{Benedikt Hurle}
 \affil{\small Chern Institute of Mathematics, Nankai University, Tianjin (P.R. China) }
 \date{May 24, 2019}

\begin{document}

\maketitle

\begin{abstract}
    In this paper we study symmetric Leibniz and related algebras, namely symmetric dialgebras and symmetric Perm-algebras. We also calculate their Koszul duals, if not known. This will give us Lie-admissible algebras and new types of algebras, which we call commutative and associative admissible algebras. Finally we  prove that all operads describing the  considered algebras are Koszul.
\end{abstract}

% AMS  17D25 (Lie admissbel algbras   , ) 18D50 (operads) 17A32 (Leibniz algebras)

\section*{Introduction}

Leibniz algebras are a non-symmetric generalization of Lie algebras. They have been studied in recent years notable by Loday \citep{MR1252069}, who also introduced there Koszul dual called Zinbiel algebras.
It is also known that the corresponding operads are Koszul.
Recently a new type of algebras called symmetric Leibniz algebras have been considered \citep{benayadi_bileib} to study Leibniz bialgebras. A symmetric Leibniz algebra is an algebra such that the multiplication satisfies the left and right Leibniz identity, which implies that the commutator is a Lie algebra.
In this paper we want to show that the corresponding operad describing symmetric Leibniz algebras is Koszul and calculate its Koszul dual.
We also will find a sequence of operads similar to~\cite{chapoton} which involves symmetric Leibniz algebras instead of Leibniz algebras, and similarly for the other involved types of algebras.
The sequence considered in loc.cit. is given by $\operad{Leib} \to \operad{DiAss} \to \operad{Perm}$, where the operads describe Leibniz, diassociative and Perm-algebras resp.
The sequence can be obtained by taking the Manin white product of the sequence $\Lie\to \Ass \to \Com$ with $\operad{Perm}$.
The sequence we will construct involves symmetric variants of these algebras and can also be obtained using the Manin white product.
We will also consider the Koszul duals of these symmetric algebras which we call admissible algebras, since the only previously considered operad that appears is $\operad{LieAdm}$, which describes Lie admissible algebras.
Finally we will prove that the operads we describe are Koszul. For Lie-admissible algebras this has already been done in \citep{remm2}.

The paper is structured as following: Frist we recall the definitions of the  types of algebras  and their relations which can be found in \citep{MR1860994,chapoton}.

In \cref{sc:symalg} we recall the definition of symmetric Leibniz and dialgebras from \citep{bordemannleib}, as well as symmetric Perm-algebras, which were called 3-abelian algebras in \citep{remm}.
We also show how they can be obtained using the Manin white or Hadamard product.

In \cref{sc:admalg} we calculate the Koszul dual of the operads defined in the previous section which we call admissible algebras.
We also obtain a commuting diagram of operads.

In \cref{sc:adjoint} we comment on some of the adjoint functors to the functors induced on the categories of algebras by the operad morphisms described in the previous sections.
Finally in \cref{sc:koszul} we proof that the symmetric types of algebras we considered are Koszul and briefly remark on their cohomology theory.

\section*{Acknowledgements}

The author wants to thank S. Benayadi and M. Bordemann for helpful discussions and E. Remm for pointing out  \cite{remm2}.
This work was partially supported by the  LPMC of Nankai University.

\section{Preliminaries}\label{sc:prel}

We want to recall the definition of different types of algebras related to Leibniz algebras which we will need in the following and maybe are not widely known. We assume that the reader is slightly familiar with language of operads, see e.g. \citep{operads,dotsenko1}. 
Given an operad $\OP$ we denote by $\OP^!$ its Koszul dual operad, and by $\OP(n)$ the $S_n$-module describing the operations of arity $n$. Further $\Alg{\OP}$ denotes the category of $\OP$-algebras. 
We denote the operads describing associative, commutative and Lie algebras by $\Ass$, $\Com$ and $\Lie$ resp. 
We will consider all algebras and operads to be over some fixed field $\K$ of characteristic zero.

There is a sequence of operads $\Lie \to \Ass \to \Com$ where the first map is induced by taking the commutator, such that the resulting map $\Alg{Com} \to \Alg{Lie}$ is zero, and an associative algebra is commutative if and only if the corresponding Lie algebra is trivial, i.e. the bracket is zero.

In \citep{chapoton,MR1860994} a similar sequence for Leibniz algebras is constructed. We recall the definition of the involved operads. Note that we consider left Leibniz algebras instead of right Leibniz algebras.

\begin{defn}
    A (left) Leibniz algebra is a vector space $L$ with a linear map $\mu: L \otimes L \to L, a \otimes b \mapsto a \cdot b$, such that for all $a,b,c \in L$
    \begin{align}
        a \cdot (b \cdot c) = (a \cdot b) \cdot c + b \cdot (a \cdot c).
    \end{align}
\end{defn}
This means the left multiplication with a fixed element is a derivation for every element.
Similarly one can define right Leibniz algebras.
We denote the operad describing left (resp. right) Leibniz algebras by $\lLeib$ (resp. $\rLeib$).
The free Leibniz algebra is as  vector space isomorphic to the tensor algebra but the product is different. So the operad as $S$-module is isomorphic to $\Ass$.

\begin{defn}
    A dialgebra (or diassociative algebra) is a vector space $A$ with two multiplications   $\vdash ,\dashv: A \otimes A \to A $ which are both associative and satisfy for all $a,b,c \in A$
    \begin{align}
        (a \vdash b ) \dashv c & =   a \vdash( b \dashv c), \\
        a \dashv (b \vdash c)  & = a \dashv b \dashv c,     \\
        (a \dashv b ) \vdash c & = a \vdash b \vdash c.
    \end{align}
\end{defn}
We will denote the corresponding operad by $\operad{DiAss}$.

There is map $\lLeib \to \mathbf{DiAss}$ given by the fact, that for a diassociative algebra $a \cdot b = a \vdash b - b \dashv a$ is a Leibniz algebra. Also every  associative algebra can be regarded as a dialgebra where the two multiplications are the same.

Now one can, similar to the case of Lie and associative algebras, consider the diassociative algebras such that the corresponding Leibniz algebra is trivial. This leads to
\begin{defn}
    A Perm-algebra (or commutative dialgebra) is  a vector space $A$ with an associative multiplication $\cdot: A \otimes A \to A$, such that
    \begin{equation}
        (a \cdot b) \cdot c = (b \cdot a )\cdot c.
    \end{equation}
\end{defn}

Perm-algebras were introduced by Chapoton \citep{chapoton}.

Since a Perm-algebra is an associative algebra it can also be considered as a dialgebra. Also clearly every commutative algebra is a Perm-algebra.

Also including the Koszul dual this gives the following commuting diagram, see \citep{chapoton}:

\begin{equation} \label{eq:leibdiag}
    \begin{tikzpicture}
        \node(a) {$\Lie$};
        \node(b)[right=of a] {$\Ass$};
        \node(c)[right=of b] {$\Com$};
        \node(d)[above=of a] {$\lLeib$};
        \node(e)[above=of b] {$\mathbf{DiAss}$};
        \node(f)[above=of c] {$\mathbf{Perm}$};
        \node(g)[below=of a] {$\mathbf{PreLie}$};
        \node(h)[below=of b] {$\mathbf{Dend}$};
        \node(i)[below=of c] {$\mathbf{Zinb}$};
        \draw[->] (a) edge (b) (b) edge (c) (d) edge (e) (e) edge (f)
        (d) edge (a) (e) edge (b) (f) edge (c)
        (g) edge (h) (h) edge (i) (a) edge (g) (b) edge (h) (c) edge (i);
    \end{tikzpicture}
\end{equation}

We want to recall the definition of the other appearing algebras next.

\begin{defn}\label{de:prelie}
    A pre-Lie algebra is a vector space $A$ with a multiplication $ \circ : A \otimes A \to A$ such that
    \begin{equation}
        (x \circ y ) \circ z -   x \circ( y  \circ z) =   (y \circ x ) \circ z - y \circ( x  \circ z).
    \end{equation}
\end{defn}
Note that the commutator defines a Lie bracket. Also every associative algebra is a pre-Lie algebra.
We denote the describing operad by $\operad{PreLie}$. It is the Koszul dual operad to $\operad{Perm}$.

The free pre-Lie algebra can be found in \citep{MR1827084}. There it also proven that the operad $\operad{PreLie}$ is Koszul.
In \citep{MR2763748} it is shown that the symmetric part of the multiplication is magmatic, i.e. it satisfies no relations. However there are relations between the symmetric part and the induced Lie bracket.

Similar, as Lie algebras can be obtained from associative algebras, one can obtain pre-Lie algebras from dendriform algebras, which are defined as follows:

\begin{defn}
    A dendriform algebra is a vector space $A$ with two multiplications $ \prec, \succ : A\otimes A \to A$ such that
    \begin{align}
        ( x \prec y) \prec z & =  x \prec (y \prec z) +  x \prec (y \succ z),  \\
        ( x \succ y) \prec z & =  x \succ (y \prec z),                         \\
        x \succ (y \succ z)  & =  (x \succ y) \succ z +  (x \prec y) \succ z.
    \end{align}
\end{defn}
The operad $\operad{Dend}$ describing dendriform algebras is the Koszul dual of $\operad{DiAss}$.

Given a dendriform algebra the map $a \circ b = a \succ b - b \prec a$ is a pre-Lie structure.
Also one obtains an associative algebra by setting $a b =   a \succ b +   a \prec b$.

Now again one can check when the induced pre-Lie algebra is trivial, and obtains:
\begin{defn}
    A Zinbiel (or commutative dendriform) algebra, is a vector space $Z$, with an multiplication $Z \otimes Z \to Z$ such that
    \begin{equation}
        x \cdot ( y \cdot z ) = (x \cdot y ) \cdot z + (y \cdot x) \cdot z.
    \end{equation}
\end{defn}
So every Zinbiel algebra can be considered as a dendriform algebra with $x \succ y :=  y \prec x := x \cdot y$.
Also $x \cdot y + y \cdot x$ defines a commutative associative multiplication.

Zinbiel algebras were introduced as dual Leibniz algebras \citep{MR1379265,MR1860994}.

There is a functorial way of obtaining the sequence $\lLeib \to \operad{DiAss} \to \operad{Perm}$ by taking the Manin white product of  $\Lie \to \Ass \to \Com$ with $\Com$.
This agrees in this case with taking the Hadamard product. Similarly the lower sequence in Diagram~  \ref{eq:leibdiag} can be obtained by taking the Manin black product with $\operad{PreLie}$.   See \citep{chapoton} for a proof of this.

We recall briefly some basics on Manin black and white products of binary quadratic operads, which we will need later. 
A definition and good introduction on Manin products in general can e.g. be found in \citep{vallettemanin}.

Given two operads $\OP, \OO$ we denote the Manin white product by $\OO \circ \OP$ and the black product by $\OO \bullet \OP$. The two products are related by Koszul duality and we have
\begin{equation}\label{th:manindual}
    (\OO \circ\OP)^! = \OO^! \bullet \OP^!.
\end{equation}

Further we recall when the white product agrees with the Hadamard product, see \citep[Sect.4.1]{vallettemanin}. 
The Hadamard product $\OO \otimes_H \OP$ of two operads $\OO$ and $\OP$ is defined as $(\OO \otimes_H \OP)(n) = \OO(n) \otimes \OP(n)$ where the tensor product is the tensor product of $S_n$-modules. It is again an operad.
Let $V$ be an $S$-module concentrated in arity 2 and $\F(V)$ be the free operad on $V$.
 This can be seen as binary trees with vertices   decorated by operations of $V$. 
For every tree $T$ with $n-1$ vertices  there is a map $L_T :  V^{\otimes (n-1)} \to \F(V)(n)$.
\begin{prop}[{\citep[Prop. 15]{vallettemanin}}]\label{th:hadamanin}
    Let $\OP$ be a binary quadratic operad generated by $V$ such that for every $n \geq 3$ and every binary
    tree $T$ with $n - 1$ vertices, the composite $\pi \circ L : V^{\otimes (n-1)} \to \F(V)(n) \to \OP(n)$ is surjective.
    Then for every binary quadratic operad $\OO$, the white product $\OP \circ \OO$ is equal to the Hadamard product $P \otimes_H Q$.
\end{prop}
Note that all operads which we consider here are binary and quadratic. The condition of the previous proposition is e.g. satisfied for $\Com$, which is also the neutral element of the Manin white product, and $\operad{Perm}$.

\section{Symmetric algebras}\label{sc:symalg}

In this section we want to define symmetric variants of Leibniz, diassociative and Perm-algebras. Symmetric Leibniz algebras can be found in \citep{benayadi_bileib}, symmetric dialgebras in \citep{bordemannleib}. Symmetric Perm-algebras have been considered under the name 3-abelian algebras in \citep{remm}.

\begin{defn}
    A symmetric Leibniz algebra is a left Leibniz algebra $A$ which is also a right Leibniz algebra. This means the multiplication $\cdot:A\otimes A \to A$ satisfies 
    \begin{align}
        a \cdot (b \cdot c ) = ( a \cdot b ) \cdot c + b \cdot (a \cdot c), \\
        (a \cdot b  ) \cdot c  = a \cdot (b \cdot c ) + ( a \cdot c ) \cdot b.        
    \end{align}
\end{defn}
There is a different description of symmetric Leibniz algebras. For this note that  given a binary algebra one can define the commutator by $[a,b] = a\cdot b - b \cdot a$ and the  anti-commutator by $a \diamond b = a \cdot b + b \cdot a$. So giving the antisymmetric part $[,]$ and symmetric part $\diamond$ is equivalent to giving the multiplication $\cdot$.

Using this there is an equivalent description of sym. Leibniz algebras.
\begin{prop}[{\citep{benayadi_bileib}}]
    A binary algebra $A$ is a sym. Leibniz algebra if and only if $[,]$ is a Lie bracket, $\diamond$ is 2-nilpotent, i.e. $(a \diamond b) \diamond c = 0 =a \diamond( b \diamond c )$, and
    $a \diamond [b,c] = 0 = [ a\diamond b,c]$ for all $a,b,c \in A$.
\end{prop}

Using this proposition it is easy to proof that the free algebra on a vector space $V$ is given by $\Lie(V) \oplus S^2 V$ where $\Lie(V)$ denotes the free Lie algebra  and  $S V$ the symmetric algebra on $V$. Similarly for the operad describing symmetric Leibniz algebras we have $\operad{sLeib} = \Lie \oplus \K \diamond$.

The operad $\sLeib$ is the pushout of the the diagram:
\begin{center}
    \begin{tikzpicture}
        \node(a) {$\F(V)$};
        \node(b)[right=of a] {$\operad{lLeib}$};
        \node(c)[below=of a]  {$\operad{rLeib}$};
        \path[->] (a) edge (b)  (a) edge (c) ;
    \end{tikzpicture}
\end{center}

Here $V$ is the $S$-module spanned by one binary operation without symmetry and $\operad{lLeib}$ and $\operad{rLeib}$ denote the operad of left and right Leibniz algebras resp.  This follows from the fact that symmetric Leibniz algebras are obtained by requiring that they are left and right Leibniz algebras.

Now one wants to adopt the definition of dialgebra such that the induced Leibniz algebra is a symmetric Leibniz algebra. This has been considered in \citep{bordemannleib}.

\begin{defn}
    A symmetric dialgebra is a vector space $D$ with two associative multiplications $\vdash,\dashv: \D \otimes D \to  D$ such that for all $a,b,c \in D$
    \begin{equation}
        (a \circ_1 b) \circ_2 c = a \circ_3 ( b \circ_4  c).
    \end{equation}
    For $\circ_i \in \{ \vdash, \dashv  \}$.
\end{defn}

Every sym. dialgebra is clearly also a dialgebra.

This means that any composition of three or more elements $x_1, \dots ,x _n \in D$ gives the same result.
So the free algebra on a vector space $V$ is given by $TV \oplus   (V \otimes V)$.

One can introduce two new multiplications on a sym.\ dialgebra. We set
\begin{align}
    a * b &= a \vdash b + a \dashv b, \\
    a \wedge b &= a \vdash b - a \dashv b.
\end{align}

\begin{lemma}
    Given a sym.\ dialgebra, the multiplication $*$ is associative, $\wedge$ is 2-nilpotent, i.e. $a \wedge  (b \wedge c) = 0 =( a \wedge  b) \wedge c$, and they satisfy $a * (b \wedge c) = (a \wedge b)*c = (a*b) \wedge c = a \wedge (b*c) =0$.
\end{lemma}
\begin{proof}
    This is a very easy calculation.
    Since for example
    \begin{equation*}
        (a *b)*c = (a \vdash b) \vdash c +  (a \dashv b) \vdash c +  (a \vdash b) \dashv c + (a \dashv b) \dashv c = 4  (a \vdash b) \vdash c =a* (b*c).
    \end{equation*}
\end{proof}
Note that giving $*$ and $\wedge$ is equivalent to giving $\vdash$ and $\dashv$.

So the operad can be described as $\Ass \oplus \K\wedge$, where any nontrivial composition with $\wedge$ is zero.

\begin{prop}
    Given a sym.\ dialgebra the map $ a \cdot b := a \vdash b  - b \dashv a$ defines  a sym.\ Leibniz algebra. 
\end{prop}
\begin{proof}
    It is clear that it is a left Leibniz algebra. It also clear that the right Leibniz identity is satisfied since it involves three elements and so it reduces to the fact the the commutator of an associative algebra is a Lie algebra.   
\end{proof}

Now we want to define  another type of algebras. It should be symmetric dialgebras such that the induced sym.\ Leibniz algebra is zero. So we impose the additional relation $a \dashv b = b \vdash a$. This leads to
\begin{defn}
    A sym.\ Perm-algebra is a vector space $A$ with an associative multiplication $\cdot: A \otimes A \to A$ such that for all $a,b,c \in A$
    \begin{equation}
        a \cdot b \cdot c =  a \cdot c \cdot b = b \cdot a \cdot c.
    \end{equation}
\end{defn}
This has already be considered in \citep{remm}, since the dual operad is given by the operad of Lie-admissible algebras. There sym.\ Perm-algebras were called 3-abelian algebras.

This implies that all products involving three or more elements are commutative.

So the free algebra on a vector space $V$ is given by $SV \oplus \Lambda^2 V$. The multiplication is such that $a (b \wedge c) =0$ and $(a \wedge b )\wedge c = 0 =a \wedge( b \wedge c)$.

The operad $\operad{sPerm}$ describing sym.\ Perm-algebras is given by $\Com \oplus \K []$. The symmetric part $\diamond$ is a commutative associative multiplication.

Every sym. Perm-algebra is also a Perm-algebra. Further every sym. Perm-algebra can be regarded as a sym. dialgebra with $a \vdash b = b \dashv a = a \cdot b$.

We get a sequence of operads $\sLeib \to \operad{sDiAss} \to \sPerm$.
This sequence can also be obtained from the sequence $\Lie \to \Ass \to \Com$ by using the Manin white or Hadamard product.
For this we first note:

\begin{lemma}\label{th:spermanin}
    The operad $\sPerm$ satisfies the condition of \cref{th:hadamanin}.
\end{lemma}
\begin{proof}
    Since for $n>2$ the $S_n$-module   $\sPerm(n)$ agrees with $\Com(n)$, this is clear.
\end{proof}

\begin{prop}\label{th:leibmanin}
    We have $\sLeib = \Lie \circ \sPerm$, $\operad{sDiAss} = \Ass \circ \sPerm$ and $\sPerm = \Com \circ \sPerm$.
\end{prop}
\begin{proof}
    The last statement is clear, since $\Com$ is the neutral element for the Manin white product.
    For the other two we use \cref{th:hadamanin} and \cref{th:spermanin}. The Hadamard product can be computed arity by arity. Since $\sPerm(n)= \Com(n)$,  $\sLeib(n) = \Lie(n) $, $\operad{sDiAss}(n) = \Ass(n)$  for $n > 2$ and $\Com$ is the neutral element for the Hadamard product, we only need to check $\sLeib(2) = (\Lie \otimes_H \sPerm)(2)$, $\operad{sDiAss}(2) = (\Ass  \otimes_H \sPerm)(2)$,
    which is clear.
\end{proof}

Most of the relations between the operads we defined so far can be summarized in the following commutative diagram:
\begin{center}
   \begin{tikzpicture}
    \node(a) {$\Lie$};
    \node(b)[right=of a] {$\Ass$};
    \node(c)[right=of b] {$\Com$};
    \node(d)[above=of a] {$\sLeib$};
    \node(e)[above=of b] {$\mathbf{sDiAss}$};
    \node(f)[above=of c] {$\mathbf{sPerm}$};
    \node(g)[above=of d] {$\mathbf{Leib}$};
    \node(h)[above=of e] {$\mathbf{DiAss}$};
    \node(i)[above=of f] {$\mathbf{Perm}$};
    \draw[->] (a) edge (b) (b) edge (c) (d) edge (e) (e) edge (f)  (g) edge (h) (h) edge (i)
    (d) edge (a) (e) edge (b) (f) edge (c)
    (g) edge (d) (h) edge (e) (i) edge (f);
    \draw[dashed](b)  edge[->] (i);
\end{tikzpicture} 
\end{center}
%$ ([xshift=5]a.north) edge[->] ([xshift=5]d.south) 
We have not shown that the diagram is commutative, but this is in all cases easy to see.

We note that in the diagram we consider left Leibniz and Perm algebras in the top row, clearly there is a similar diagram which involves right Leibniz and Perm-algebras.

Further note that the middle (top) row can be obtained by taking the Manin white product of the bottom row with $\operad{sPerm}$ ($\operad{Perm}$ resp.). We  show the related maps as solid arrows.

\section{Admissible algebras}\label{sc:admalg}

We now want to give and calculate if necessary the Koszul dual operads of the operads considered in the last section.

We start with symmetric Leibniz algebras. We use the description in term of a symmetric operation $\diamond$ and an antisymmetric operation $[,]$.
So we have $\sLeib = \factor{\F(V)}{(R)}$, where $\F(V)$ is the free operad on $V = \operatorname{span}(\lievert2 , \comvert2 ) $ and $(R)$ is the ideal generated by $R =\operatorname{span}( \tree{lie}{com},  \tree{com}{lie}, \sum_{cycl.} \tree{lie}{lie} )$. Here $\lievert2$ corresponds to the bracket and $\comvert2$ to the symmetric part $\diamond$.

We now compute the orthogonal complement $R^\bot$ with respect to the dual pairing  given in \citep[Sect. 7.6.3]{operads}. We denote the dual of $\lievert2$ by $\comvert2$ and vice versa.
Since $\dim(\F(V)(2))=12$ and dim $R= 10$ we have $\dim R^\bot=2$ and it is spanned by $\tree{com}{com} - (123) \tree{com}{com}, \tree{com}{com} - (321) \tree{com}{com}$.
These two relations are clearly orthogonal to all relations involving $\comvert2$ and are known to be orthogonal to the relation  describing Lie algebras since the Koszul dual of $\operad{Lie}$ is $\operad{Com}$. So $\diamond$ is a commutative associative multiplication.

Of course one can now again describe to resulting operad using one non-symmetric operation and we define:

\begin{defn}
    A com.-admissible algebra $C$ is a vector space with a multiplication $C \otimes C \to C$ written as  $a \otimes b \mapsto a \cdot b $ such that the symmetric part $a \diamond b = a \cdot b + b \cdot a$ is a commutative associative multiplication.  
\end{defn}

The operad $\operad{ComAdm}$ describing com.-admissible algebras is by our construction the Koszul dual of $\sLeib$. It seems difficult to find a good description other then by its defining generators and relations.

So obviously every com.-admissible gives rise to a commutative algebra. Also every Zinbiel algebra is a com.-admissible algebra. This gives a sequence of operads $\operad{Com} \to \operad{ComAdm} \to \operad{Zinb}$ which is Koszul dual to the series  $\operad{Leib} \to \operad{sLeib} \to \operad{Lie}$.

We will see later in \cref{th:kos}  that $\operad{sLeib}$ is Koszul. Since one knows the dimension of the space $\sLeib(n)$  for each arity $n$, this allows us to calculate the dimension of $\sLeib^!(n)= \operad{ComAdm}(n)$. For this we use the fact that $f^{\sLeib^!}(-f^{\sLeib}(x)) = -x$, where $f^\OP$ is the generating series of $\OP$, given by $f^\OP =\sum_{i=1}^\infty \dim \OP(n)
    \frac{x^n}{n!}$ for an arbitrary binary quadratic operad $\OP$, see \citep[7.6.12.]{operads}.
We have $f^{\sLeib}(x) = f^{\Lie}(x) + \frac{x^2}{2} = -\log(1-x) + \frac{x^2}{2}$, since  $\operad{sLeib} = \Lie \oplus \K \diamond$ and it is well known that  $f^{\Lie}(x) = -\log(1-x)$. Using computer software one can calculate to first terms of $f^{\operad{ComAdm}}$ and obtains
\begin{align*}
    f^{\operad{ComAdm}}(x) & = 1 x +2 \frac{x^2}{2!} + 10 \frac{x^3}{3!} + 86 \frac{x^4}{4!} + 1036\frac{x^5}{5!} + 16052\frac{x^6}{6!}+ 304060\frac{x^7}{7!} \\
                        & + 6807656\frac{x^8}{8!} + 175881016\frac{x^9}{9!} + 5150163272\frac{x^{10}}{10!} + \dots.
\end{align*}

%for Assadm [b1=1,b2=4,b3=42,b4=744,b5=18480,b6=590400,b7=23058000,b8=1064367360,b9=56693831040,b10=3422589811200]
%
%for LieAdm [b1=1,b2=2,b3=11,b4=101,b5=1299,b6=21484,b7=434314,b8=10376729,b9=286071990,b10=8938291341]
%
%for Comadm  [b1=1,b2=2,b3=10,b4=86,b5=1036,b6=16052,b7=304060,b8=6807656,b9=175881016,b10=5150163272]
%

Another way of obtaining com.-admissible algebras is considering left and right Zinbiel algebras. 
Left Zinbiel algebras can be described, using a generator and relations as $\operad{Zinb} =\factor{\F(V)}{(R_l)}$ where $V$ is the $S_2$-module described by one binary operation, and $R_l \subset \F(V)(3)$ is the space of relations. 
Similarly  $\operad{rZinb} =\factor{\F(V)}{(R_r)}$. Now one can consider the operad $\factor{\F(V)}{(R_l \cap R_r)}$, which turns out to be $\operad{ComAdm}$. This is dual to the fact that sym. Leibniz algebras can be obtained by combining the relations for left and right Leibniz algebras.

Next we want  to consider the Koszul dual of sym.\ dialgebras.

Again one can use both of the two equivalent descriptions to compute the Koszul dual. We will use the one by $*$ and $\wedge$.
We have $\dim \F(V)(2) = 8$ if considered as a non-symmetric operad with $V =\spa( \verte{2}{*},\verte{2}{$\wedge$} )$ and
$\operad{sDiAss}= \factor{\F(V)}{(R)}$ with 
\begin{equation}
   R = \spa( \treete{*}{*} - \treerte{*}{*} , \treete{$\wedge$}{*}, \treerte{$\wedge$}{*} , \treete{$\wedge$}{$\wedge$}, \treerte{$\wedge$}{$\wedge$} , \treete{*}{$\wedge$}, \treerte{*}{$\wedge$}  ). 
\end{equation}

We denote the dual of $*$ again by $*$ and similarly for $\wedge$.
Since $\dim(R)=7$ we have $\dim( R^\bot) =1$ and $R^\bot = \spa(\treete{*}{*} - \treerte{*}{*} )$.
This is clearly orthogonal to all relations and by dimensional reasons it is the complete orthogonal complement.

\begin{defn}
    An associative admissible algebra $A$  is a vector space with two operations $\succ, \prec : A \otimes A \to A$ such that $ a *b = a \succ b + a \prec b$ is an associative multiplication.
    We will denote the corresponding operad by $\operad{AssAdm}$.
\end{defn}

By the previous we have $\operad{sDiAss}^! = \operad{AssAdm} $.

By the very definition every ass.-admissible algebra defines an associative algebra. Also every dendriform algebra is an ass.-admissible algebra by the obvious map. It is also clear that any com.-admissible algebra with multiplication $\cdot$ can be  considered as an  ass.-admissible with $a \succ b = b \prec a = a \cdot a$.

Using the Koszulness of $\operad{AssAdm}$ and $f^{\operad{sDiAss}}(x) = f^{\Ass} (x) + x^2 = \frac{x}{1-x} + x^2$, one can calculate the generating series of $\operad{AssAdm}$ and with this the dimension of $\operad{AssAdm}(n)$ for small $n$. We obtain
\begin{align*}
    f^{\operad{AssAdm}}(x) & = 1 x +4 \frac{x^2}{2!} + 42 \frac{x^3}{3!} + 744 \frac{x^4}{4!} + 18480\frac{x^5}{5!} + 590400\frac{x^6}{6!}+ 23058000\frac{x^7}{7!} \\
                        & + 1064367360\frac{x^8}{8!}  + 56693831040\frac{x^9}{9!} + 3422589811200\frac{x^{10}}{10!} + \dots
\end{align*}

As mentioned earlier the Koszul dual of $\sPerm$ is given by $\operad{LieAdm}$, the operad describing Lie-admissible algebras. Lie-admissible and related  algebras, such that the commutator defines a Lie bracket, were studied in \cite{remm,remm2}. There also the Koszul dual was calculated.

\begin{defn}
    A Lie admissible algebra is a vector space $L$ with a multiplication $L \otimes L \to L: a \otimes b \mapsto a \circ b$ such that the commutator $[a,b] =a\circ b -b \circ a$ is a Lie bracket.
\end{defn}
Again it is clear that a pre-Lie algebra is a Lie-admissible algebra and every Lie-admissible gives rise to a Lie algebra.

The operad $\operad{LieAdm}$ can be described as a quotient of the free operad in one symmetric operation $\diamond$ and one anti-symmetric operation $[,]$ in arity 2 with the relation that $[,]$ is a Lie bracket. 
From this it is clear that there is map from the operad $\operad{ComMag}$ describing magmatic commutative algebras to $\operad{LieAdm}$, which is injective.  Using the remark after \cref{de:prelie} we get a sequence $\operad{ComMag} \to \operad{LieAdm} \to \operad{PreLie}$. 

Is is again possible to compute the beginning of the generating series $f^\operad{LieAdm}$ and, using $f^\operad{sPerm}=\exp(x)-1 + \frac{x^2}{2} $, we obtain
\begin{align*}
    f^{\operad{AssAdm}}(x) & = 1 x +2 \frac{x^2}{2!} + 11 \frac{x^3}{3!} + 101 \frac{x^4}{4!} + 1299\frac{x^5}{5!} + 21484\frac{x^6}{6!}+ 434314\frac{x^7}{7!}+ 10376729\frac{x^8}{8!} \\
                        & + 286071990\frac{x^9}{9!} + 8938291341\frac{x^{10}}{10!} + \dots
\end{align*}

Also given an $\operad{AssAdm}$-algebra one can construct a Lie-admissible algebra structure by setting $a\circ b= a \succ b - b \prec a$. Then one gets for the induced Lie bracket $[a,b] =a \circ b - b \circ a = a * b - b * a$. This is a Lie bracket since $*$ is associative.

Finally an ass.-admissible algebra such that $\circ$ as defined before is zero can be seen as a com.-admissible algebra, since in this case $a * b= a \succ b + a\prec b = a \succ b + b \prec a = b *a$.

As a corollary to \cref{th:leibmanin} and \cref{th:manindual} we get:
\begin{corollary}
    We have $\operad{LieAdm} =\Lie \bullet \operad{LieAdm}$, $\operad{AssAdm} =\Ass \bullet \operad{LieAdm}$ and $\operad{ComAdm} =\Com \bullet \operad{LieAdm} $.
\end{corollary}

Including the Koszul duals (and omitting the dashed lines and their duals) the diagram at the end of the last section becomes:

\begin{center}
    \begin{tikzpicture}
        \node(a) {$\Lie$};
        \node(b)[right=2 of a] {$\Ass$};
        \node(c)[right=2 of b] {$\Com$};
        \node(d)[above=of a] {$\sLeib$};
        \node(e)[above=of b] {$\mathbf{sDiAss}$};
        \node(f)[above=of c] {$\mathbf{sPerm}$};
        \node(d2)[above=of d] {$\lLeib$};
        \node(e2)[above=of e] {$\mathbf{DiAss}$};
        \node(f2)[above=of f] {$\mathbf{Perm}$};
        \node(g)[below=of a] {$\mathbf{LieAdm}$};
        \node(h)[below=of b] {$\mathbf{AssAdm}$};
        \node(i)[below=of c] {$\mathbf{ComAdm}$};
        \node(g2)[below=of g] {$\mathbf{PreLie}$};
        \node(h2)[below=of h] {$\mathbf{Dend}$};
        \node(i2)[below=of i] {$\mathbf{Zinb}$};
        \draw[->] (a) edge (b) (b) edge (c) (d) edge (e) (e) edge (f)  (g) edge (h) (h) edge (i)
        (d2) edge (e2) (e2) edge (f2)  (g2) edge (h2) (h2) edge (i2)
        (d2) edge (d) (d) edge (a)  (a) edge (g) (g) edge (g2)
        (e2) edge (e) (e) edge (b)  (b) edge (h) (h) edge (h2)
        (f2) edge (f) (f) edge (c)  (c) edge (i) (i) edge (i2);
    \end{tikzpicture}    
\end{center}

\section{Adjoint functors}\label{sc:adjoint}

We want to mention that every map of operads gives a functor in the opposite direction on the corresponding categories of algebras. So given a map $\OO \to \OP$ there is a functor $\Alg{\OP} \to \Alg{\OO}$. This functor always has an adjoint functor.
We want to  comment on some of this adjoint functors for the maps in the previous diagram.

For $\Alg{\Lie} \to \Alg{\Ass}$ to functor is given by the universal enveloping algebra as is commonly known.  Similar for (sym.) Leibniz one can construct universal enveloping (symmetric) dialgebras, see  \citep{MR1860994} (resp. \citep{bordemannleib}).

In principal one can also construct an enveloping dendriform algebra for a pre-Lie algebra $(P,\circ)$, by considering the free dendriform algebra, which e.g. can be found in \citep{MR1860994}, and factoring out the ideal generated by $x \succ y - y \prec x - x \circ y$ for $x,y \in P \subset \operad{Dend} (P)$.
Similarly one could define an enveloping ass.-admissible algebra for a Lie-admissible algebra. However since we do not know a good description for the free ass.-admissible algebra, we will not consider this further.

Also it is known that for an associative algebra $A$ the abelianization given by $\factor{A}{([A,A])}$ is a commutative algebra. This is the adjoint functor to the forgetful functor $\Alg{Com} \to \Alg{Ass}$.  Similarly given a dialgebra $D$, one can consider the quotient $\factor{D}{(D \circ D )}$ where $\circ$ is the induced Leibniz algebra. This quotient is by construction a Perm-algebra. Similar constructions can be done for the other horizontal arrows on the right.

It also well known that to any Leibniz algebra $(L,\cdot)$ one can associate a Lie algebra, by quotienting the ideal generated by $x \cdot x$ for $ x \in L$.
Given a sym. Leibniz algebra one can  regard it as a left (or right) Leibniz algebra and then consider the associated Lie algebra. Note that this in general is not the  Lie algebra given by the commutator since the underlying vector space can be different.
It is also possible to associate a symmetric Leibniz algebra to a Leibniz algebra $L$. It is given by the quotient of $L$ with respect to the ideal generated by $(x \cdot y) \cdot z - x \cdot ( y \cdot z ) - (x\cdot z) \cdot y$ for $x,y,z \in L$.

\section{Koszulness and cohomology}\label{sc:koszul}

In this section we want to prove that the sym.\ and admissible types of  algebras we considered are Koszul.
For Lie-admissible algebras this is already known and has been proven in \citep{remm2}. For sym.\ Leibniz and sym.\ dialgebra to our knowledge this is not know and also there  Koszul duals were not  known.

\begin{theorem}\label{th:kos}
    The operads $\sLeib, \operad{sDiAss}$ and $\operad{sPerm}$ are Koszul.
\end{theorem}
Before we prove this we note that by Koszul duality wee also get:
\begin{corollary}
    The operads $\operad{LieAdm}$, $\operad{AssAdm}$ and $\operad{ComAdm}$ are Koszul.
\end{corollary}
\begin{proof}
    We will only give the proof for sym. Leibniz algebras, since the other two case are very similar.

    We will use the rewriting method e.g.\ described in \citep[Sect. 8.3]{operads} or \citep{dotsenko_grobner}.
    We will use the description of $\sLeib$ using $[,]$ and $\diamond$. Then any critical monomial  either only involves $[,]$ and belongs actually to $\Lie$ so it is known to be confluent, or it involves $\diamond$ and at least one other operation. In this case the rewriting will always give zero. So it is also confluent.
\end{proof}

So in principal this can be used to define cohomologies for the algebras considered. Since given a Koszul operad $\OO$ the coderivations on the connected cofree coalgebra over the Koszul dual cooperad give a Lie algebra such that the Maurer-Cartan elements are $\OO$-algebras. 

In the case of Lie-admissible algebras this agrees with the cohomology and graded  Lie bracket on the cochain complex  given in \citep{remm}.

Since the free sym. Leibniz and sym.\ di-algebras -- and with this also the connected cofree coalgebra -- are easy to understand we get a cohomology for com.- and ass.-admissible algebras. This cohomologies are closely related the cohomology of the induced commutative and associative algebra resp. In the associative case for example  in the cochain complex one just has  to add  $\Hom(A \otimes A,A)$   to $\Hom(A \otimes A,A)$.

In principal one also obtains cohomologies for sym.\ Leibniz and sym.\ dialgebras but in these cases the underlying complex is quite big and difficult to describe since the free com.- and ass.-admissible algebras are quite  big.

\bibliographystyle{bibstyle}
\bibliography{bibli}

\end{document}